\documentclass[a4paper,12pt]{article}
\usepackage{amssymb}
\usepackage{bbm}
\usepackage[reqno]{amsmath}
\usepackage{amsthm}
\usepackage[utf8]{inputenc}
\usepackage{hyperref}
\usepackage{enumitem}
\usepackage{color}

\newcommand{\IR}{\ensuremath{\mathbb{R}}}
\newcommand{\IN}{\ensuremath{\mathbb{N}}}
\newcommand{\IZ}{\ensuremath{\mathbb{Z}}}

\newcommand{\IP}{\ensuremath{\mathbb{P}}}
\newcommand{\IE}{\ensuremath{\mathbb{E}}}

\renewcommand{\rho}{\varrho}

\newcommand{\err}{\Delta}
\newcommand{\edet}{\Delta^{\rm unif}}

\newcommand{\norm}[1]{\left\Vert#1\right\Vert}

\newcommand{\abs}[1]{\left|#1\right|}
\newcommand{\brackets}[1]{\left(#1\right)}

\definecolor{darkgreen}{RGB}{40, 138, 80}

\DeclareMathOperator{\dist}{dist}

\DeclareMathOperator{\INT}{INT}

\newcommand{\vol}{\mathrm{vol}}
\newcommand{\dd}{ {\rm d}}
\DeclareMathOperator*{\esssup}{ess\,sup}

\newtheorem{thm}{Theorem}
\newtheorem{cor}{Corollary}

\theoremstyle{plain}
\newtheorem{lemma}{Lemma}
\newtheorem{prop}{Proposition}
\theoremstyle{definition}

\newtheorem{alg}{Algorithm}
\newtheorem{rem}{Remark}

\makeatletter
\newcommand{\leqnomode}{\tagsleft@true}
\newcommand{\reqnomode}{\tagsleft@false}
\makeatother

\title{Recovery of Sobolev functions\\ restricted to iid sampling}

\author{David Krieg\footnote{Institut f\"ur Analysis, 
JKU Linz, Austria,
\texttt{david.krieg@jku.at}, \texttt{mathias.sonnleitner@jku.at}.
},
\,Erich Novak\footnote{Mathematisches Institut, FSU Jena, Germany, \texttt{erich.novak@uni-jena.de}.}
\ and Mathias Sonnleitner$^*$}

\date{\today}

\begin{document}

\maketitle

\begin{abstract}
We study $L_q$-approximation and integration for functions 
from the Sobolev space $W^s_p(\Omega)$ and compare optimal 
randomized (Monte Carlo) algorithms with algorithms 
that can only use iid sample points, uniformly distributed on the domain. 
The main result is that we obtain the same optimal rate 
of convergence if we restrict to iid sampling, 
a common assumption in learning and uncertainty quantification. 
The only exception is when $p=q=\infty$, where 
a logarithmic loss cannot be avoided. 
\end{abstract}

\medskip

\centerline{\begin{minipage}[hc]{130mm}{
}
\end{minipage}}
\medskip

\section{Introduction and main results} 

Let $\Omega \subset \IR^d$ be open and bounded.
We assume that $\Omega$ satisfies an interior cone condition.
We study the problem of approximating
a function $f$ from the Sobolev space $W_p^s(\Omega)$
in the $L_q(\Omega)$-norm based on function values $f(x_j)$ 
on a finite set of sampling points $P=\{x_1,\hdots,x_n\}$.
This makes sense if $s > d/p$, 
in which case $W_p^s(\Omega)$ is compactly embedded
into the space of bounded continuous functions~$C_b(\Omega)$, 
and in the case $p=1$ and $d=s$.  
We also study the problem 
of numerical integration. 

\smallskip

There is a vast literature on the error for optimal sampling points. 
It is known that the rate of convergence of the 
worst case error of optimal 
deterministic algorithms is
\[
 n^{-s/d+(1/p-1/q)_+}
\]
for the approximation problem and
\[
 n^{-s/d}
\]
for the integration problem, where $a_+:=\max\{a,0\}$, $a\in\mathbb{R}$.
These are classical results for special domains like the cube,
see Ciarlet~\cite[Chapter~3]{Cia78} and Heinrich~\cite[Section~6]{Hei94}. 
For general domains, we refer to Narcowich, Wendland and 
Ward \cite{NWW04} as well as Novak and Triebel \cite{NT06}, 
see also Remark~\ref{rem:UB}. 

\smallskip

In this paper we assume that we cannot choose the sampling points.
Instead they are given to us as realizations 
of independent random variables
which are uniformly distributed on the domain.
That is, we get our data $f(x_j)$ for random sample points $x_j\in \Omega$
which are not under our control.

\smallskip

Then one can still consider the ``uniform'' or ``worst case'' 
error on the unit ball of $W_p^s(\Omega)$, 
as done by two of the authors in~\cite{KS20}. 
There it is proved that, in expectation (and thus with high probability),
the worst case error of
random points is asymptotically optimal 
for $L_q$-approximation whenever $q<p$; 
otherwise $n$ random points behave
as well as $n/\log n$ optimal points. 

\smallskip
 
But it is natural to study also a different approach and error criterion.  
If $P \subset \Omega$ is random, then also an algorithm is 
a randomized or Monte Carlo algorithm and one may compare 
the algorithm with optimal randomized algorithms
based on 
the expected  error
for inputs from the unit ball of $W_p^s(\Omega)$. 
So we switch from the expected worst case error
to the worst case expected error.

\smallskip
 
The class of all randomized algorithms is very large 
since one may construct the random variable $x_{k+1}$ 
based on the realization 
$(x_1, f(x_1), \dots , x_k, f(x_k))$
in a sophisticated and complicated way. 
Nevertheless, 
using this weaker notion of error, 
one cannot improve the rate  
$
 n^{-s/d+(1/p-1/q)_+}
$ 
for the approximation problem, 
see Math\'e~\cite{Ma91} and Heinrich~\cite{He08}. 
For the integration problem one now obtains the improved 
order
$n^{-s/d-1/2}$ if $p \ge 2$ 
and 
$n^{-s/d-1+1/p}$ if $1 \le p <2$, 
see Bakhvalov~\cite{Ba62} and Novak~\cite{No88}.
 
\smallskip
 
In this paper we check what we can get with randomized (or Monte Carlo) algorithms 
when we restrict to independent uniformly distributed sampling. 
We will prove that we still obtain the optimal order 
of convergence for the approximation problem 
unless $p=q=\infty$: In this limiting case 
there is a logarithmic loss. 
For the integration problem we obtain the optimal order 
for all $p$. 

\medskip

We now describe our results in more detail. 
We assume that $\Omega \subset \IR^d$ is open and bounded
and satisfies an interior cone condition.
That is, there is some $r>0$ and $\theta \in (0,\pi]$ 
such that for every $x\in\Omega$, we find a cone 
\[
K(x):=\{x+\lambda y: y\in\mathbb{S}^{d-1}, \langle y,\xi(x)\rangle\ge\cos\theta,\lambda\in [0,r] \}
\]
with apex $x$, direction $\xi(x)$ in the unit sphere $\mathbb{S}^{d-1}$, 
opening angle $\theta$ and radius $r$
such that $K(x) \subset \Omega$. 
Here, $\langle \cdot, \cdot\rangle$ denotes the standard inner product.
For $s\in \IN$ and $1\le p \le \infty$
such that $s > d/p$ or $p=1$ and $s=d$, we consider the Sobolev space
\[
W_p^s(\Omega) \,:=\, \left\{ f \colon \Omega \to \IR  \ \big\vert\ D^\alpha f 
 \in L_p(\Omega) \text{ for all } \alpha \in \IN_0^d \text{ with } |\alpha| \le s  \right\}
\]
with semi-norm
\[
 |f|_{W_p^s(\Omega)} \,:=\, \left(\sum_{|\alpha| = s} \Vert D^\alpha f \Vert_{L_p(\Omega)}^p\right)^{1/p}
\]
and norm 
\[
 \Vert f\Vert_{W_p^s(\Omega)} := 
 \left(\sum_{|\alpha| \le s} \Vert D^\alpha f \Vert_{L_p(\Omega)}^p\right)^{1/p},
\]
with the usual modification for $p=\infty$.
Note that $W_p^s(\Omega)$ is continuously embedded into $C_b(\Omega)$, 
the space of bounded continuous functions 
with the sup norm, and that the embedding is compact
in the case $s>d/p$, see e.g.\ Maz'ya \cite[Section~1.4]{Maz11} 
for this fact on domains satisfying an interior cone condition.
We study the problem of $L_q(\Omega)$-approximation ($1\le q \le \infty$)
on $W_p^s(\Omega)$. 
\smallskip

For a random operator $A\colon W_p^s(\Omega) \to L_q(\Omega)$
we define the Monte Carlo error
\[
\err\left(A,W_p^s(\Omega),L_q(\Omega)\right)
\,:=\,
\sup_{\Vert f\Vert_{W_p^s(\Omega)} \le 1}\, \IE\ \left\Vert f - A(f) \right\Vert_{L_q(\Omega)}.
\]
Later we will study a measurable algorithm $A$, where the expectation exists, but we 
may use the notation $\IE$ for any mapping if we take the upper integral
for the definition.  
Given a random point set $P\subset \Omega$, we put
\[
\err\left(P,W_p^s(\Omega),L_q(\Omega)\right)
\,:=\,
\inf_A\, \err\left(A,W_p^s(\Omega),L_q(\Omega)\right),
\]
where the infimum is taken over all random operators 
of the form $A(f)=\varphi(f|_P)$ 
with a random mapping $\varphi\colon \IR^P \to L_q(\Omega)$. 
Note that $f|_P=\big(f(x)\big)_{x\in P}$ is the restriction 
of $f$ to the point set $P$ and we use this as our information.
This is the smallest Monte Carlo error that can be achieved with the
sampling point set $P$.
Moreover, we put
\[
\err\left(n,W_p^s(\Omega),L_q(\Omega)\right)
\,:=\,
\inf_P\, \err\left(P,W_p^s(\Omega),L_q(\Omega)\right),
\]
where the infimum is taken over all random point sets
of cardinality at most~$n$.
This is the smallest Monte Carlo error that can be achieved with
$n$ optimally chosen sampling points.
It is known, at least for special domains $\Omega$,  that
\begin{equation}\label{eq:UBOPTran}
 \err\left(n,W_p^s(\Omega),L_q(\Omega)\right) \,\asymp\, n^{-s/d+(1/p-1/q)_+},
\end{equation}
see Math\'e~\cite{Ma91} and Remark~\ref{rem1}.
The symbol $\asymp$ means that the left hand side is bounded 
from above by a constant multiple of the right hand side for all $n\in\IN$ and vice versa;
we use $\preccurlyeq$ and $\succcurlyeq$ for the one-sided relations.
However, it is often not possible to choose the random sampling points to our liking.
Here, we are interested in the smallest Monte Carlo error
which can be achieved with $n$ independent and uniformly distributed
samples. The main result of our paper is the following.

\begin{thm}\label{thm:main}
 Let $\Omega \subset \IR^d$ be open and bounded, satisfying an interior cone condition
 and let $1 \le p,q \le \infty$ and $s\in \IN$ such that $s>d/p$ or $p=1$ and $s=d$.
 For every $n\in \IN$, let $P_n$ be a set of $n$ independent and uniformly distributed points 
 on~$\Omega$.
 Then
 \[
 \err\left(P_n,W_p^s(\Omega),L_q(\Omega)\right)
 \,\asymp\,
 \begin{cases}
\, \displaystyle \left(n / \log n \right)^{-s/d} 
& \text{if } p=q=\infty,\\
\, n^{-s/d+(1/p-1/q)_+} 
& \text{else}. \vphantom{\Big|}
\end{cases}
\]
\end{thm}

This means that independent and uniformly distributed points 
are (asymptotically) as good as optimally selected 
(deterministic or random) sampling points
in all cases except $p=q=\infty$.

\medskip

This answers the question for the power of independent uniformly distributed samples
with respect to the Monte Carlo error criterion.
On the other hand, one might also be interested 
in a stronger uniform error criterion.
For a random operator $A\colon W^{s}_{p}(\Omega)\to L_{q}(\Omega)$ 
the uniform error can be defined by
\[
	\edet\left(A,W^{s}_{p}(\Omega),L_{q}(\Omega)\right)
	\,=\, \IE \sup_{\norm{f}_{W^{s}_{p}(\Omega)}\le 1}\norm{f-A(f)}_{L_{q}(\Omega)}.
\]
Note that the order of the supremum and the expected value is interchanged.
Thus, while a small Monte Carlo error $\err$ means
that for every individual function, 
the error is small with high probability, 
a small uniform error $\edet$ means 
that with high probability
the error is small for every function. 
As before, given a random point set $P\subset \Omega$, we put
\[
\edet\left(P,W_p^s(\Omega),L_q(\Omega)\right)
\,:=\,
\inf_A\, \edet\left(A,W_p^s(\Omega),L_q(\Omega)\right),
\]
where the infimum is taken over all random operators 
of the form  $A(f)=\varphi(f|_P)$ 
with a random mapping $\varphi\colon \IR^P \to L_q(\Omega)$.
This is thus the smallest uniform error that can be achieved
with the random sampling point set $P$.
Moreover, we put
\[
\edet\left(n,W_p^s(\Omega),L_q(\Omega)\right)
\,:=\,
\inf_P\, \edet\left(P,W_p^s(\Omega),L_q(\Omega)\right),
\]
where the infimum is taken over all random point sets
of cardinality at most $n$. 
It is known, at least for special domains, that 
\begin{equation}\label{eq:UBOPTdet}
 \edet\left(n,W_p^s(\Omega),L_q(\Omega)\right)
 \,\asymp\, n^{-s/d+(1/p-1/q)_+},
\end{equation}
see again \cite{NWW04,NT06}
and Remark~\ref{rem:UB}.

\smallskip

For the uniform error of independent uniformly distributed samples,
the following result has been obtained in~\cite{KS20} for bounded convex domains.
See also Ehler, Graef and Oates~\cite{EGO19}
and the survey \cite{HKNPUsurvey} for earlier results in this direction.

\begin{thm}\label{thm:main-det}
 Let $\Omega \subset \IR^d$ be open and bounded, satisfying an interior cone condition
 and let $1 \le p,q \le \infty$ and $s\in \IN$ such that $s>d/p$ or $p=1$ and $s=d$.
 For every $n\in \IN$, let $P_n$ be a set of $n$ independent and uniformly 
 distributed points on~$\Omega$.
 Then
 \[
  \edet\left(P_n,W_p^s(\Omega),L_q(\Omega)\right)
 \,\asymp\, 
 \begin{cases}
\, \displaystyle \left( n / \log n \right)^{-s/d+1/p-1/q}  
& \text{if } q\ge p,\\
\ n^{-s/d}
& \text{if } q<p. \vphantom{\Big|}
\end{cases}
\]
\end{thm}

This means, with respect to the uniform error criterion,
independent and uniformly distributed samples are 
as good as optimally chosen sampling points if and only if $q<p$.
Moreover, if we are bound to independent and uniformly distributed samples,
one can achieve the same rate for the uniform error as for the Monte Carlo error
if and only if $q<p$ or $p=q=\infty$.
In all other cases,
the Monte Carlo error criterion provides a speed-up in comparison
to the uniform error criterion.

\medskip

All the upper bounds of Theorems~\ref{thm:main} and \ref{thm:main-det} 
are achieved by the same algorithm.
The algorithm works for all $p$ and $q$ and up to a given 
smoothness $s$. 
We describe the algorithm in Section~\ref{s2}.

\medskip

Let us now turn to the integration problem
\[
\INT (f) = \int_{\Omega} f(x) \, \dd x\,.
\]
We use a similar notation.
For a random operator $A\colon W_p^s(\Omega) \to \IR$
we define the Monte Carlo error
\[
\err\left(A,W_p^s(\Omega), \INT \right)
\,:=\,
\sup_{\Vert f\Vert_{W_p^s(\Omega)} \le 1}\, \IE\ \left\vert \INT(f) - A(f) \right\vert.
\]
Given a random point set $P\subset \Omega$, we put
\[
\err\left(P,W_p^s(\Omega), \INT \right)
\,:=\,
\inf_A\, \err\left(A,W_p^s(\Omega), \INT \right),
\]
where the infimum is taken over all random operators 
of the form $A(f)=\varphi(f|_P)$ 
with a random mapping $\varphi\colon \IR^P \to \IR $.
This is the smallest Monte Carlo error that can be achieved with the
sampling point set $P$.
Moreover, we put
\[
\err\left(n,W_p^s(\Omega), \INT \right)
\,:=\,
\inf_P\, \err\left(P,W_p^s(\Omega), \INT \right),
\]
where the infimum is taken over all random point sets
of cardinality at most $n$.
This is the smallest Monte Carlo error that can be achieved with
$n$ optimally chosen sampling points.
It is known, at least for special domains $\Omega$,  that 
\[
 \err\left(n,W_p^s(\Omega), \INT \right) \,\asymp\, n^{-s/d+ (1/p-1/2)_+ - 1/2 }, 
\]
see Bakhvalov~\cite{Ba59,Ba62} and Novak~\cite{No88}. 
Again,  we are interested in the smallest Monte Carlo error
which can be achieved with $n$ independent and uniformly distributed
sampling points. As a corollary to Theorem~\ref{thm:main} we obtain the following.

\begin{cor}\label{cor1}
 Let $\Omega \subset \IR^d$ be open and bounded, satisfying an interior cone condition
 and let $1 \le p \le \infty$ and $s\in \IN$ such that $s>d/p$
 or $p=1$ and $s=d$.
 For every $n\in \IN$, let $P_n$ be a set of $n$ independent and uniformly distributed points 
 on~$\Omega$.
 Then
 \[
 \err\left(P_n,W_p^s(\Omega), \INT \right)
 \,\asymp\,
 n^{-s/d+ (1/p-1/2)_+ -1/2 } . 
\]
\end{cor}

This means that, with respect to the Monte Carlo error, 
independent and uniformly distributed points 
are (asymptotically) as good as optimally selected 
(deterministic or random) sampling points
in all cases. 

\medskip

This answers the question for the power of independent uniformly distributed samples
for numerical integration
with respect to the Monte Carlo error criterion.
Again, one might also be interested 
in a stronger uniform error criterion.
For a random operator $A\colon W^{s}_{p}(\Omega)\to  \IR $ 
the uniform error can be defined by
\[
	\edet\left(A,W^{s}_{p}(\Omega),  \INT \right)
	\,=\, \IE \sup_{\norm{f}_{W^{s}_{p}(\Omega)}\le 1}  \vert \INT(f) - A(f) \vert .
\]
Again the order of the supremum and the expected value is interchanged.
As before, given a random point set $P\subset \Omega$, we put
\[
\edet\left(P,W_p^s(\Omega), \INT \right)
\,:=\,
\inf_A\, \edet\left(A,W_p^s(\Omega), \INT \right),
\]
where the infimum is taken over all random operators 
of the form $A(f)=\varphi(f|_P)$ 
with a random mapping $\varphi\colon \IR^P \to \IR $.
This is thus the smallest uniform error that can be achieved
with the random sampling point set $P$.
Moreover, we put
\[
\edet\left(n,W_p^s(\Omega),  \INT \right)
\,:=\,
\inf_P\, \edet\left(P,W_p^s(\Omega),  \INT  \right),
\]
where the infimum is taken over all random point sets
of cardinality at most $n$. 

\medskip

For the uniform error of independent uniformly distributed samples,
the following result has been obtained in~\cite{KS20} for bounded convex domains.

\begin{cor}\label{cor2}
 Let $\Omega \subset \IR^d$ be open and bounded, satisfying an interior cone condition
 and let $1 \le p \le \infty$ and $s\in \IN$ such that $s>d/p$
 or $p=1$ and $s=d$.
 For every $n\in \IN$, let $P_n$ be a set of $n$ independent and uniformly distributed points
 on~$\Omega$.
 Then
 \[
  \edet\left(P_n,W_p^s(\Omega), \INT \right)
 \,\asymp\, 
 \begin{cases}
\, \displaystyle \left( n / \log n \right)^{-s/d}  
& \text{if } p=1 ,\\
\, n^{-s/d}
& \text{if } p>1 . \vphantom{\bigg|}
\end{cases}
\]
\end{cor}

This means, with respect to the uniform error criterion,
independent and uniformly distributed samples are 
as good as optimally chosen sampling points 
if and only if $p>1$.
We also notice that for independent and uniformly distributed samples,
the Monte Carlo error criterion provides a speed-up in comparison
to the uniform error criterion for all $p$. 

\begin{rem}  \label{rem1} 
Different authors study different domains and possibly even 
different function spaces $W^s_p(\Omega)$. 
All the standard definitions of the Sobolev spaces 
coincide if $\Omega$ is a bounded Lipschitz domain;
if the domain is not Lipschitz then different texts 
possibly use different spaces. 
If the Sobolev space on $\Omega$ is defined by restriction 
of the functions from $W^s_p(\IR^d)$ then one 
obtains smaller spaces.
Our results stay true for this altered definition 
since the lower bounds
already hold for functions with compact support inside~$\Omega$.

\end{rem}

\begin{rem}\label{rem:UB}
The upper bound for optimal points in~\eqref{eq:UBOPTdet},
and therefore also the upper bound in~\eqref{eq:UBOPTran},
is known at least for Lipschitz domains.
We refer to \cite{NWW04} and \cite{NT06}.
Here we consider more general domains
and obtain \eqref{eq:UBOPTdet} as a byproduct, see Remark~\ref{rem:optimal-points-proof}.
Hence, also for optimal points, we may have slightly more general results 
compared to the existing literature.
\end{rem}

\begin{rem}\label{rem:LB}
The lower bound for optimal points in \eqref{eq:UBOPTran}, 
and therefore also the lower bound in \eqref{eq:UBOPTdet}, 
is known for the
whole class of domains that we consider.
In older texts the lower bounds are usually proved for 
Sobolev functions with a support inside the cube $[0,1]^d$ 
and then it is rather easy to prove the same lower bound 
for all domains that contain a cube, i.e., for all open sets. 
Because of a re-scaling we thus obtain different constants but 
the same order of convergence. 
The same is true for the lower bounds for random points
as given in Theorem~\ref{thm:main-det} and Corollary~\ref{cor2}.
They are known for functions with support in a cube from \cite{KS20}
and immediately transfer to the class of domains
considered here.
More generally,
these lower bounds hold for \emph{all}
(non-empty) bounded open sets
and even the boundedness is not needed.
When we speak about uniformly distributed random points 
then, of course, we assume that $\Omega$ has a finite (Lebesgue) measure.
\end{rem}

\begin{rem}
Our techniques 
can most likely 
be used to prove similar results
for more general function spaces of isotropic smoothness 
like Triebel-Lizorkin or Besov spaces as well as Sobolev spaces on 
manifolds.
Another interesting family of spaces are function spaces
of mixed smoothness as surveyed in~\cite{DTU16}.
Here, we are still quite far from understanding
the power of 
independent and uniformly distributed sampling points,
and even the power of optimal sampling points
is not known in many cases.
There are recent results in this direction for 
the special case of $L_2$-approximation 
on the Hilbert space ${\bf W}_2^s(\mathbb T^d)$
of functions with mixed smoothness $s$ on the $d$-torus.
Namely, it is known that independent and uniformly distributed sampling points 
are optimal with respect to the Monte Carlo error~\cite{Kri19}
and optimal up to a logarithmic factor with respect to the uniform error~\cite{KU19,Ull20}.
See also 
\cite{NW12,NSU20,Tem20,KU21,CD21}
for related results.
This also implies that they are optimal up to logarithmic factors
for the problem of integration on ${\bf W}_2^s(\mathbb T^d)$
with respect to both error criteria.
We do not know whether the logarithmic loss 
can be avoided
with uniformly distributed samples,
as it is the case with isotropic smoothness.
Known optimal sampling points for the integration problem
on ${\bf W}_2^s(\mathbb T^d)$ have a very particular structure,
see e.g.\ \cite[Section~8.5]{DTU16} and the references therein
for the uniform error 
and \cite{KN17,Ull17} for the Monte Carlo error.
\end{rem}

\begin{rem}
In numerical analysis and information-based complexity 
we usually want to find the optimal algorithm, based 
on the optimal point set 
$P = \{ x_1, \dots , x_n \} \subset \Omega$. 
Often it is not easy to find the optimal point set 
and 
even more often it is not possible to choose $P$, we simply have to use 
the 
information as it comes in. 
It is a very common assumption in learning theory and 
uncertainty quantification that the information comes in randomly, 
given by independent and identically distributed (iid) samples, 
see Berner, Grohs, Kutyniok and Petersen~\cite{BGKP21}, 
Giles~\cite{Gi15}, Lugosi and Mendelson~\cite{LM19},
Shalev-Shwartz and Ben-David~\cite{SB14},
Steinwart and Christmann~\cite{SC08} and Zhang~\cite{Zh20}.
In the framework of information-based complexity,
the power of iid information was recently studied and surveyed by
Hinrichs, Krieg, Novak, Prochno and Ullrich~\cite{HKNPUellipsoids,HKNPUsurvey},
by Huber and Jones~\cite{HJ19} and by Kunsch, Novak and Rudolf~\cite{KNR19}. 
Although we think that it is most natural to assume
that the iid samples are uniformly distributed on the domain,
it might also be interesting to study other distributions.
The results of this paper stay valid,
if the samples are iid with respect to a probability measure $\mu$
on $\Omega$ that satisfies $\mu(B) \ge c\, {\rm vol}(B)$ 
for some constant $c>0$ and all balls $B$ with center 
in $\Omega$ and radius smaller than a constant.
But we believe that this condition is not necessary and
it would be interesting to have a characterization of all measures
for which iid information is optimal for $L_q(\Omega)$-approximation
on $W_p^s(\Omega)$.
\end{rem}

\begin{rem} 
The main results of this paper are about the optimal order of convergence 
if we may use only iid sample points and we prove that one can 
achieve the same or almost the same order that one  can achieve with 
optimal sample points. 
There are other ways to compare the power of different algorithms 
and, for example, one could study tractability properties 
of algorithms that are based on iid sample points. 

It would be good to know more about the difference between optimal 
and iid sample points, depending on the problem, the dimension $d$ 
and the domain $\Omega \subset \IR^d$. 
One possibility is to study the asymptotic constants, such as 
\[
 \lim_{n \to \infty}  
 \err\left(P_n,W_p^s(\Omega),L_q(\Omega)\right)
\cdot 
 n^{+s/d-(1/p-1/q)_+} .
\]
We conjecture that all these asymptotic constants exist and 
possibly they do not depend on the shape of $\Omega$,
only on its volume. 
The asymptotic constants are known only in very rare cases, see \cite{No20}, though. 
Here we present an example from \cite{HKNPUsurvey} 
that nicely shows the quality of iid samples.
We study $L_1$-approximation for functions from the class 
$$
F_d = \{ f: [0,1]^d \to \IR \mid 
|f(x)-f(y)| \le d(x,y) \},
$$
with the maximum metric on the $d$-torus, i.e., 
$$
d(x,y) = \min_{k \in \IZ} \Vert x+k-y \Vert_\infty .
$$
Then
\[
\lim_{n \to \infty} 
\edet\left(P_n,F_d,L_1(\Omega)\right)
\cdot n^{1/d} = 
\frac{1}{2} \, \Gamma (1+ 1/d) \approx  \frac{1}{2} - \frac{\gamma}{2d} 
\]
with the Euler number $\gamma \approx 0.577$, noting that $\Gamma'(1)=-\gamma$. 
This compares very well with the error of optimal methods that, for $n=m^d$,  equals 
$\frac{d}{2d+2} \, n^{-1/d}$. 
To achieve the same error $\varepsilon$ in high dimension 
with random iid sample points, we have to multiply the number 
of optimal sample points by roughly $\exp(1-\gamma)\approx 1.526$;
this factor is quite small and does not 
increase with $d$. 
\end{rem}

\section{The algorithm}  \label{s2} 

Before we can formulate the algorithm, we state a result on polynomial reproduction 
using moving least squares.
We use the following result of Wendland~\cite{W01}, see also Chapter~4 in his book~\cite{W04};
in Lemma 7 of \cite{KS20} it is shown that one can skip
the well-separatedness of the points that appears in
Wendland's formulation. The covering radius of a finite point set $P\subset \IR^d$ with respect to a bounded set $\Omega\subset \IR^d$ will be denoted by
\[
h_{P,\Omega} \,:=\, \sup_{x\in\Omega}\, \dist(x,P),
\]
where $\dist(x,P)=\min_{y\in P}\|x-y\|_2$.

\begin{samepage}
\begin{lemma}
\label{lem:Wendland}
For every $\theta \in (0,\pi/5]$ and $s,d\in\IN$ 
there are constants $c_1,c_2>0$ such that the following holds.
For every cone $K\subset \IR^d$ with opening angle $\theta$ 
and radius $\varrho>0$
and every finite point set $P\subset K$ with covering radius $h_{P,K} \le c_1 \varrho$,
there are continuous functions $u_x \colon K \to \IR$ for $x\in P$ 
such that the linear operator
\[
 S_{P,K}\colon  C_b (K) \to C_b (K), 
 \quad
 S_{P,K}f = \sum_{x\in P} f(x) u_x
\]
is bounded with operator norm at most~$c_2$
and
equals the identity when restricted to
the space of polynomials of degree at most~$s$. 
\end{lemma}
\end{samepage}

Implicitly, we used that cones themselves satisfy an interior 
cone condition, see e.g. Proposition 3.13 in \cite{W04}, which we formulate as a lemma. 
Note that we can assume without loss of generality that $\Omega$ 
satisfies a cone condition with opening angle $\theta\in (0,\pi/5]$.
\begin{lemma}\label{lem:conecone}
Every cone with opening angle $\theta\in (0,\pi/5]$ and 
radius $\varrho>0$ satisfies an interior cone condition 
with parameters $r'=\frac{3}{4}c_{\theta}\varrho$ and 
$\theta'=\theta$, where $c_{\theta}=\frac{\sin \theta}{1+\sin \theta}$.
\end{lemma}

We add another lemma which will ensure that our algorithm is well defined.
Recall that
$\Omega \subset \IR^d$ is open and bounded and 
satisfies an interior cone condition with parameters $r>0$ 
and $\theta \in (0,\pi]$.
That is, for every $x\in\Omega$ we find a cone $K(x)$
with apex~$x$, opening angle $\theta$ and radius $r$
such that $K(x) \subset \Omega$.
Without loss of generality, we may assume that $\theta\le \pi/5$
and that the direction $\xi(x)$ of the cone $K(x)$
depends continuously on the apex $x$ for almost all $x\in\Omega$, 
see Lemma~\ref{lem:ccones}.
We write $K(x,\varrho)$ for the 
cone with radius $\varrho>0$ and the same vertex, direction and opening angle as $K(x)$. Then $K(x,\varrho)$ is contained in the closure of the Euclidean ball 
$B(x,\varrho):=\{y\in\IR^d:\|x-y\|_2<\varrho\}$ intersected with $\Omega$. 
\smallskip

\begin{lemma}~\label{lem:welldef}
Let $c_0=c_{\theta}c_{1}/2$ with $c_\theta$ as in Lemma~\ref{lem:conecone}
and $c_1$ as in Lemma~\ref{lem:Wendland}.
If $P\subset \Omega$ satisfies $h_{P,\Omega} < c_0 r$, 
then we have $h_{P\cap K(x),K(x)} \le c_1 r$
for all $x\in\Omega$.
\end{lemma}

\begin{proof}
Assume for a contradiction that there is some $x\in\Omega$ and some $y\in K(x)$
such that the ball $B(y,c_1 r/2)$ is empty of $P\cap K(x)$.
Using the cone condition satisfied by $K(x)$,
we obtain a ball $B(z,c_0 r)$ 
that is contained in $B(y,c_1 r/2)\cap K(x)$.
See
\cite[Lemma 2]{KS20}, 
a direct consequence of  \cite[Lemma 3.7]{W04}, 
for details.
The ball $B(z,c_0 r)$ is thus empty of $P$, which contradicts $h_{P,\Omega} < c_0 r$.
Note that in order to apply Lemma~2 from \cite{KS20} the inequality $c_{1}\le \frac{3}{4}c_{\theta}$ was used implicitly, but this is justified since we may choose the constant $c_1$ in Lemma~\ref{lem:Wendland} small enough.
\end{proof}

We now define the algorithm $A_P\colon  C_b(\Omega) \to B(\Omega)$,
which uses samples on an $n$-point set $P\subset \Omega$.
Here, $B(\Omega)$ is the space
of bounded functions with the supremum norm.
The algorithm is independent of $p$ and $q$ and works for any smoothness up to $s$. 
We define the algorithm
for general point sets $P$, 
although it will only be used with random points.

\smallskip

\begin{alg}\label{alg}
Let $\Omega \subset \IR^d$ be open and bounded,
satisfying an interior cone condition. Choose
constants $r>0$ and $\theta \in (0,\pi/5]$
and an almost everywhere continuous function $\xi\colon \Omega \to \mathbb S^{d-1}$
such that for all $x\in\Omega$ and $\rho\le r$ 
the cone $K(x,\rho)$ with apex~$x$, 
opening angle $\theta$, radius $\rho$ 
and direction $\xi(x)$ is contained in $\Omega$.
Let $s\in\IN$ and $c_0=c_0(s,r,\theta,d)$ be the constant from Lemma~\ref{lem:welldef}.
For all natural numbers $n$ with $n\ge r^{-d}$ 
and all point sets $P \in \Omega^n$, 
we define the algorithm $A_P\colon  C_b(\Omega) \to B(\Omega)$ as follows.
We put $m_0:=\lfloor \log_2 (rn^{1/d}) \rfloor$ 
and define $m_P(x)$ for $x\in\Omega$
as the maximal integer $m\in \{0,\ldots,m_0\}$ 
such that the cone
$K(x,2^{-m}r)$ satisfies
\begin{equation} \label{eq:local-covering}
	h_{P\cap K(x,2^{-m}r),K(x,2^{-m}r)}\, \le \, c_1 \, 2^{-m}\,r.
\end{equation}
If no such cone exists, we take $m_P(x)=0$.
We set $r_P(x):=2^{-m_P(x)}r$ and $K_P(x):=K(x,r_P(x))$. 
We distinguish two scenarios.

\smallskip

\textbf{Scenario~1}: \ $h_{P,\Omega}\ge c_0 r$.
Then, for all $x\in\Omega$ and $f\in C_b(\Omega)$, we put
\[
A_P f (x) \,:=\, 
	0. 
\]

\textbf{Scenario~2}: \ $h_{P,\Omega} < c_0 r$.
Then, for all $x\in\Omega$ and $f\in C_b(\Omega)$, we put
\[
A_P f (x) \,:=\, 
	S_{P\cap K_P(x),K_P(x)}f(x).
\]

With $m_n(x)$, $r_n(x)$, $K_n(x)$ and $A_nf(x)$
we denote the random quantities obtained from choosing
$P$ as a set of $n$ random points, 
independently and uniformly distributed in $\Omega$.
\end{alg}

Let us briefly comment on the structure of the algorithm. 
In Scenario~1, there is a large region
where we do not have any information about the target function.
The sampling points are not suited to achieve an error smaller than a constant.
We therefore return a trivial output.
For random points, Scenario~1 is exponentially unlikely.
In Scenario~2, Lemma~\ref{lem:welldef} ensures that for every $x\in\Omega$,
there is a cone with apex~$x$, opening angle $\theta$
and direction $\xi(x)$ such that sufficiently many points
reside in this cone in order to apply the polynomial
reproducing map from Lemma~\ref{lem:Wendland}.
The cone $K_P(x)$ with radius $r_P(x)$ is the smallest cone with this property.
The condition $m\le m_0$ leads to
$r_P(x)\ge n^{-1/d}$ 
and imposes
no essential restriction as this is asymptotically the quantity we would expect from an optimally distributed point set.
The algorithm computes at each point $x$ an approximation based on the information 
given by the samples inside the cone $K_P(x)$ and one can compute the value $A_Pf(x)$ numerically using an implementation of moving least squares.

\smallskip

For each point set $P$,
the algorithm $A_P$ maps linearly from $C_b(\Omega)$ to $B(\Omega)$
and its operator norm is bounded by the constant $c_2$
from Lemma~\ref{lem:Wendland}.  
Further, it follows from the underlying construction in \cite{W04} that, 
for fixed $f$, the output $A_Pf(x)$
is a measurable function of $(x,P)\in\Omega^{n+1}$.
For this and other measurability issues, we refer the reader to Section~\ref{sec:meas}. 
We will show the following.

\begin{thm}\label{thm:alg}
Algorithm~\ref{alg} obeys the following error bounds.
\begin{align}
	\err\left(A_n,W_p^s(\Omega),L_q(\Omega)\right) \,&\preccurlyeq\, 
 \begin{cases}
\, \displaystyle (n/ \log n)^{-s/d} & \text{if } p=q=\infty,\\[11pt]
\, n^{-s/d+(1/p-1/q)_+} & \text{else}. \vphantom{\bigg|}
\end{cases}
	\label{eq:ran-old}\\
\edet\left(A_n,W_p^s(\Omega),L_q(\Omega)\right)\,&\preccurlyeq\,
  \begin{cases}
\, \displaystyle (n/\log n)^{-s/d+1/p-1/q} & \text{if } q\ge p,\\[11pt]
\, n^{-s/d} & \text{else}. \vphantom{\bigg|}
\end{cases}
	\label{eq:det-old}\\
	\nonumber
\end{align}
\end{thm}

The next section forms the proof of these upper bounds.
The corresponding lower bounds,
which complete the proof of Theorems~\ref{thm:main} and \ref{thm:main-det},
are known, see Remark~\ref{rem:LB}, except for the case $p=q=\infty$, which will be proven in Section~\ref{sec:lower}.
\medskip

\section{Proof of Theorem~\ref{thm:alg}}

In Scenario~2, we get the following point-wise estimate for the error.

\begin{lemma}
\label{lem:pointwise}
There is a constant $c_3>0$ such that
for all $f \in W_p^s(\Omega)$ and all $x\in \Omega$, we have
in Scenario~2 that
\[
 \left| f (x) - A_n f (x) \right| \,\le\, c_3\, r_n(x)^{s-d/p}\, |f|_{W_p^s(K_n^\circ(x))}.
\]
\end{lemma}

\begin{proof}
	Using an affine map to a reference cone $K$, 
	together with the continuous embedding of $W^s_p(K^\circ)$ into the continuous functions and the generalized Poincar\'e inequality from \cite[Ch.~1.1.11]{Maz11} we find a constant $c_4>0$ independent of $f, x$ and $n$ 
and a polynomial $\pi$ of degree at most $s$ such that
\[
\Vert f - \pi \Vert_{L_\infty(K_n(x))}
\,=\, \Vert f - \pi \Vert_{L_\infty(K_n^\circ(x))} 
\,\le\, c_4\, r_n(x)^{s-d/p}\, |f|_{W_p^s(K_n^\circ(x))}.
\]
Here, $K_n^\circ(x)$ is the interior of $K_n(x)$.
By Lemma~\ref{lem:Wendland},
\begin{multline*}
 \left| f (x) - A_n f (x) \right| 
 \,\le\, \left\Vert f - S_{P_n\cap K_n(x),K_n(x)} f \right\Vert_{L_\infty(K_n(x))} 
 \\
 \,\le\, \Vert f - \pi \Vert_{L_\infty(K_n(x))} 
 + \Vert S_{P_n\cap K_n(x),K_n(x)}(f-\pi) \Vert_{L_\infty(K_n(x))} \\
 \,\le\, (1+c_2) \, \Vert f - \pi \Vert_{L_\infty(K_n(x))},
\end{multline*}
which proves the statement.
\end{proof}

In Scenario~1, where we have $A_n = 0$, 
we clearly get the estimate
\[
 \left\Vert f - A_nf \right\Vert_{L_q(\Omega)}
 \,\le\, \vol(\Omega)^{1/q} \left\Vert f \right\Vert_{C_b(\Omega)}
 \,\le\, c_5 \left\Vert f \right\Vert_{W_p^s(\Omega)}.
\]
Moreover, the probability of Scenario~1 is exponentially small in $n$. 
This can be shown by a simple net argument.
We do not have to prove it since 
the probability estimate follows from \cite[Theorem~2.1]{RS16} for
$\alpha=c N/\log N$ for a suitable constant $c>0$.
Together with
Lemma~\ref{lem:pointwise}, 
this implies
that Theorem~\ref{thm:alg} follows
once we prove the following estimates:

\begin{equation}\label{eq:ran-new}
 \sup_{\Vert f\Vert_{W_p^s(\Omega)} \le 1} \IE\, 
 \left\Vert r_n(x)^{s-d/p}\, |f|_{W_p^s(K_n^\circ(x))} \right\Vert_{L_q(\Omega)} 
 \preccurlyeq\, 
 \begin{cases}
\, \displaystyle (n/ \log n)^{-s/d} & \text{if } p=q=\infty,\\[11pt]
\, n^{-s/d+(1/p-1/q)_+} & \text{else}. \vphantom{\bigg|}
\end{cases}
\end{equation}
\begin{equation}\label{eq:det-new}
 \IE\, \sup_{\Vert f\Vert_{W_p^s(\Omega)} \le 1}\, 
 \left\Vert r_n(x)^{s-d/p}\, |f|_{W_p^s(K_n^\circ(x))} \right\Vert_{L_q(\Omega)}
 \,\preccurlyeq\, 
  \begin{cases}
\, \displaystyle (n/\log n)^{-s/d+1/p-1/q} & \text{if } q\ge p,\\[11pt]
\, n^{-s/d} & \text{else}. \vphantom{\bigg|}
\end{cases}
\end{equation}
\smallskip

To prove these estimates,
we have to study the random variable $r_n(x)$ for $x\in \Omega$,
which is the radius of the approximation cone at $x$.
In fact, it will be helpful to have upper bounds on the radius
of the largest approximation cone containing a certain point $y\in\Omega$.
We therefore introduce
\begin{equation}
\label{eq:defrnast}
 r_P^*(y) \,:=\, \esssup_{x\in \Omega}\, r_P(x)\,1(y\in K_P(x))
\end{equation}
and the random variable $r_n^*(y)$ obtained
if $P$ is a set of $n$ independent and uniformly distributed points.
Note that both $r_n(y)$ and $r_n^\ast(y)$ take values between $n^{-1/d}$ and $r$.
We will prove the following.
\smallskip
\begin{prop}\label{prop:localradius}
	For all $y\in \Omega$ and $\alpha\ge 0$, we have
\[
 \IE\, r_n(y)^\alpha \,\asymp\, 
 \IE\, r_n^*(y)^\alpha \,\asymp\, 
 n^{-\alpha/d},
\]
where the implied constants are independent of $y$ and $n$.
\end{prop}

Our proof strategy for Proposition~\ref{prop:localradius} involves suitably scaled grids
which we now define. 
If $Q(y,\varrho):=\{x\in\IR^d: \max_{1\le i\le d}|x_i-y_i|\le \varrho\}$ is a (closed) cube around a point $y\in\Omega$ and $\ell\in\IN$, then we partition $Q(y,\varrho)$ into $\ell^{d}$ equally sized cubes of sidelength $2\varrho\ell^{-1}$. We collect all of the cubes which are fully contained in $\Omega$ into a set $\mathrm{Grid}(y,\varrho,\ell)$. Obviously, its cardinality $\#\mathrm{Grid}(y,\varrho,\ell)$ can be at most $\ell^d$.

\begin{lemma}\label{lem:emptycube}
Let $\ell:=\lceil\frac{8\sqrt{d}}{c_{\theta}c_{1}}\rceil$, where $c_{\theta}$ is as in Lemma~\ref{lem:conecone}.
Let $y\in\Omega$ and
$\rho_n \in \{r_n,r_n^*\}$.
If $\rho_n(y)=2^{-m}r$ for some $m< m_0$, 
then there is a cube $Q\in\mathrm{Grid}(y,2^{-m}r,\ell)$ with $Q\cap P_{n}=\emptyset$.
\end{lemma}

\begin{proof}
By definition of $r_n^*(y)$ and $r_n(y)$, there is some $x\in\Omega$
such that $y\in K_n(x)$ and $r_n(x)=2^{-m} r$.
In the case $\rho_n=r_n$, we may take $x=y$.
Since $y\in K_{n}(x)$, we have $K_{n}(x)\subset Q(y,2^{-m}r)$.
By the definition of $r_n$, the cone $K=K(x,2^{-m-1} r)\subset K_{n}(x)$ satisfies
\[
	h_{P_n\cap K,K}>c_{1}\, 2^{-m-1} r. 
\]
That is, we find $z\in K\subset \Omega$ such that $B(z,c_{1}2^{-m-1} r)$
is empty of $P_{n} \cap K$. 
Arguing as in the proof of Lemma~\ref{lem:welldef} we obtain a ball of radius $c_{\theta}c_{1} 2^{-m-1} r$ contained in $K$ and empty of $P_{n}$. This ball contains a cube of sidelength $c\, 2^{-m-1} r$ with $c=c_{\theta}c_1/\sqrt{d}$, which in turn contains a cube of $\mathrm{Grid}(y,2^{-m}r,\ell)$ if
\[
\ell^{-1}2^{-m+2} r<c\, 2^{-m-1} r,	
\]
which is fulfilled for our choice of $\ell$. 
\end{proof}

\begin{lemma}\label{lem:rntailbound}
	There are constants $C,c>0$ independent of $n$ such that for all $y\in \Omega$ and $t>0$, and for $\rho_n \in \{r_n,r_n^*\}$ we have
\[
	 \mathbb P\left(\rho_n(y) > t\right) \,\le\, C \exp\left(- c\, t^{d}\, n\right).
\]
\end{lemma}
\begin{proof}
For $t\ge r$, the probability 
is zero and the estimate is true for any $C,c>0$.
For $t< 2^{-m_0} r$, the estimate is ensured by
taking $C\ge \exp(2^d c)$.
Thus let $2^{-m_0} r \le t <r$ from now on,
i.e., $t\in [2^{-m_1-1} r, 2^{-m_1} r)$ for some $m_1<m_0$. 
We have
\[
	\IP(\rho_n(y)>t)
	\,=\,
	\sum_{m=0}^{m_1} \IP\left(\rho_n(y)=2^{-m}r\right).
\]
By Lemma~\ref{lem:emptycube} we have
\begin{align*}
	\IP\left(\rho_n(y) = 2^{-m} r\right)
	\,&\le\, \IP\left(\exists Q\in \mathrm{Grid}(y,2^{-m}r,\ell): Q\cap P_{n}=\emptyset\right)\\
	&\le\, \#\mathrm{Grid}(y,2^{-m}r,\ell)(1-\tilde c(2^{-m}r)^d)^{n}\\
	&\le\, \ell^{d}\exp\big(-\tilde c \cdot \lfloor (2^{-m}r)^d\, n \rfloor\big)
\end{align*}
with a constant $\tilde c>0$ depending only on $\ell$ and the domain $\Omega$.
For $m\in\{0,\hdots,m_1\}$, the numbers $\lfloor (2^{-m}r)^d\, n\rfloor$
are distinct natural numbers greater than $t^d n/2$,
such that we get
\[
\IP(\rho_n(y)>t) \,\le\, 
\ell^d \sum_{k>t^d n/2}^{\infty} \exp(-\tilde c\, k)
\,\le\, C \exp(- \tilde c\, t^d n/2).
\]
\end{proof}

Now, we easily obtain Proposition~\ref{prop:localradius}.

\begin{proof}[Proof of Proposition~\ref{prop:localradius}]
	The statement is clear for $\alpha=0$ and thus consider $\alpha>0$.
	Let $\rho_n\in\{r_n,r_n^*\}$.
	The estimate
\[
 \IE\, \rho_n(y)^\alpha \,\ge\, n^{-\alpha/d}
\]
is immediate. On the other hand, Lemma~\ref{lem:rntailbound} yields
\[
 \IE\, \rho_n(y)^\alpha
 \,\le\, \int_0^\infty \mathbb P\left(\rho_n(y)^\alpha>t\right)\, \dd t
 \,\le\, C^\alpha \int_0^\infty \exp\left(- c\, t^{d/\alpha}\, n\right) \dd t 
 \,=\, C' n^{-\alpha/d}
\]
where $C':=C^\alpha \int_0^\infty \exp\left(- c\, u^{d/\alpha}\right) \dd u$.
\end{proof}

We are now ready to prove \eqref{eq:ran-new}
and \eqref{eq:det-new}
which will conclude the proof of Theorem~\ref{thm:alg}.
If not specified otherwise, $f$ will always denote an element 
of the unit ball of $W_p^s(\Omega)$.
\smallskip

\subsection{Proof of \eqref{eq:ran-new}}

We perform a case distinction in $p$ and $q$.

\subsubsection{The case $p=\infty$}
\label{subsec:pinf}

Here, we simply compute
\[
 \IE\ \left\Vert r_n(x)^{s-d/p}\, |f|_{W_p^s(K_n^\circ(x))} \right\Vert_{L_q(\Omega)}
 \,\le\, 
 \IE\, \left\Vert r_n(x)^{s} \right\Vert_{L_q(\Omega)}.
\]
In the case $q<\infty$, we arrive with Jensen and Fubini at
\[
 \le\, \left(\IE\, \left\Vert r_n(x)^{s} \right\Vert_{L_q(\Omega)}^q\right)^{1/q}
 \,=\, \left( \int_\Omega \IE\ r_n(x)^{sq} \, \dd x\right)^{1/q}
\]
and it follows by Proposition~\ref{prop:localradius} that this is dominated by $n^{-s/d}$, as desired.
In the case $q=\infty$, we use that for any realization $r_n(x) \le c_6 h_{P_n,\Omega}$ with $c_6>0$ 
independent of both $x$ 
and the point set, 
which follows from the proof of Lemma~\ref{lem:welldef},
so that 
\[
 \IE\, \left\Vert r_n(x)^{s} \right\Vert_{L_q(\Omega)}
 \,\preccurlyeq\, \IE\, h_{P_n,\Omega}^s 
 \,\preccurlyeq\, \left( \frac{\log n}{n} \right)^{s/d},
\]
where the last estimate is known e.g.\ from Reznikov and Saff~\cite[Theorem~2.1]{RS16}, 
using the cone condition of $\Omega$.
$\hfill \square$

\subsubsection{The case $p=q < \infty$}

Due to Jensen and Fubini
\begin{multline*}
\IE\ \left\Vert r_n(x)^{s-d/p}\, |f|_{W_p^s(K_n^\circ(x))} \right\Vert_{L_p(\Omega)}
\,\le\, \left( \IE\, \int_\Omega r_n(x)^{sp-d}\, |f|_{W_p^s(K_n^\circ(x))}^p \, \dd x \right)^{1/p}\\
=\, \left( \sum_{|\alpha|=s} \int_\Omega  |D^\alpha f(y)|^p  \cdot 
 \int_\Omega  
\IE \left[r_n(x)^{sp-d}\,1(y\in K_n(x))\right] \dd x \, \dd y  \right)^{1/p}.
\end{multline*}
Therefore, the desired upper bound \eqref{eq:ran-new} follows, once we prove
\begin{equation}\label{eq:inner-int}
 \int_\Omega  
\IE \left[r_n(x)^{sp-d}\,1(y\in K_n(x))\right] \dd x 
\,\preccurlyeq\, n^{-sp/d}
\end{equation}
with a constant independent of $y$.
We use H\"older's inequality and Proposition~\ref{prop:localradius}
to obtain
\begin{multline*}
\IE \left[r_n(x)^{sp-d}\,1(y\in K_n(x))\right]
\,\le\, \left[\IE\, r_n(x)^{2(sp-d)}\right]^{1/2} \cdot \mathbb P(y\in K_n(x))^{1/2}\\
\preccurlyeq\, n^{-sp/d+1} \cdot \mathbb P(y\in K_n(x))^{1/2}.
\end{multline*}
From Lemma \ref{lem:rntailbound} we deduce for any $y\in\Omega$ 
\[
	 \mathbb P(y\in K_n(x))
	 \,\le\, \IP(r_{n}(x)\ge \norm{x-y}_{2}) 
	 \,\preccurlyeq\, \exp\left(- c \Vert x-y \Vert_2^{d} \, n\right),
\]
since the cone $K_{n}(x)$ has radius $r_{n}(x)$. 
Here, implicit constants are independent of $y$. Therefore,
\[
 \int_\Omega \mathbb P(y\in K_n(x))^{1/2}\, \dd x
 \,\preccurlyeq\, \int_{\IR^d} \exp\left(- \frac{c}{2} \Vert x-y \Vert_2^{d} \, n\right) \,\dd x
 \,=\, n^{-1} \int_{\IR^d} \exp\left(- \frac{c}{2} \Vert u \Vert_2^{d}\right) \,\dd u,
\]
from which the desired estimate \eqref{eq:inner-int} follows. 
 $\hfill \square$

\subsubsection{The case $p<q=\infty$}

Again due to Jensen and Fubini and the fact that the essential supremum of the integral is smaller than the integral of the essential supremum, we compute
\begin{multline*}
\IE\ \left\Vert r_n(x)^{s-d/p}\, |f|_{W_p^s(K_n^\circ(x))} \right\Vert_{L_\infty(\Omega)}
\,\le\, \left( \IE\, \left\Vert r_n(x)^{sp-d}\, |f|_{W_p^s(K_n^\circ(x))}^p \right\Vert_{L_\infty(\Omega)} \right)^{1/p}\\
\,=\, \left( \IE\, \esssup_{x\in \Omega} \sum_{|\alpha|=s} \int_\Omega |D^\alpha f(y)|^p \cdot r_n(x)^{sp-d} \,1(y\in K_n(x))\,\dd y \right)^{1/p}\\
\le\, \Bigg( \sum_{|\alpha|=s} \int_\Omega |D^\alpha f(y)|^p \cdot 
\IE\underbrace{\left[ \esssup_{x\in \Omega}\,  r_n(x)^{sp-d} \,1(y\in K_n(x)) \right]}_{\displaystyle 
=\, r_n^*(y)^{sp-d}
}\,\dd y\Bigg)^{1/p}
\end{multline*}
and the desired estimate \eqref{eq:ran-new} is obtained from Proposition~\ref{prop:localradius}. $\hfill \square$

\subsubsection{The case $p<q<\infty$}

This follows by interpolation
from the previous cases.
Namely,
\[
 \Vert f - A_n(f) \Vert_{L_q(\Omega)} \,\le\, 
 \Vert f - A_n(f) \Vert_{L_p(\Omega)}^{p/q}
 \cdot
 \Vert f - A_n(f) \Vert_{L_\infty(\Omega)}^{1-p/q}
\]
and by H\"older's inequality with conjugates $q/p$ and $(1-p/q)^{-1}$, we get
\[
 \IE\, \Vert f - A_n(f) \Vert_{L_q(\Omega)} \,\le\, 
 \left(\IE\,\Vert f - A_n(f) \Vert_{L_p(\Omega)}\right)^{p/q}
 \cdot
 \left(\IE\,\Vert f - A_n(f) \Vert_{L_\infty(\Omega)}\right)^{1-p/q}.
\]
Inserting the corresponding upper bounds for $q=\infty$ and $q=p$,
we see that the desired bound for $q\in (p,\infty)$ follows.
$\hfill \square$
\smallskip

\subsubsection{The case $q<p<\infty$}
This follows from the case $q=p<\infty$.
$\hfill \square$
\smallskip

\subsection{Proof of \eqref{eq:det-new}}

Again, we perform a case distinction.

\subsubsection{The case $q<p=\infty$} 
Here, we observe
\[
 \IE\ \sup_{\Vert f\Vert_{W_p^s(\Omega)} \le 1}\left\Vert r_n(x)^{s-d/p}\, |f|_{W_p^s(K_n^\circ(x))} \right\Vert_{L_q(\Omega)}
 \,\le\, 
 \IE\, \left\Vert r_n(x)^{s} \right\Vert_{L_q(\Omega)}
\]
and proceed as in Section~\ref{subsec:pinf}.
$\hfill \square$
\smallskip

\subsubsection{The case $q<p<\infty$}
We choose $\delta>0$ arbitrary and split the integrands of
\[
 	\left\Vert r_n(x)^{s-d/p}\, |f|_{W_p^s(K_n^\circ(x))} \right\Vert_{L_q(\Omega)}^{q}
	= \int_{\Omega} r_n(x)^{(s-d/p)q}\, |f|_{W_p^s(K_n^\circ(x))}^{q}\, \dd x
\]
into
\[
	r_{n}(x)^{sq+\delta} 
	\quad \text{and} \quad
	r_{n}(x)^{-dq/p-\delta} |f|_{W_p^s(K_n^\circ(x))}^{q}
\]
and apply H\"older's inequality with conjugate indices $r=(1-q/p)^{-1}$ and $p/q$ to obtain
\begin{align*}
 	\big\Vert r_n(x)^{s-d/p}\,& |f|_{W_p^s(K_n^\circ(x))} \big\Vert_{L_q(\Omega)}^{q}\\
	&\le \brackets{\int_{\Omega}r_{n}(x)^{(sq+\delta)r}\dd x}^{1/r}
	\brackets{\int_{\Omega} r_{n}(x)^{-d-\delta p/q} |f|_{W_p^s(K_n^\circ(x))}^{p}\dd x}^{q/p}.
\end{align*}
The righthand integral satisfies by virtue of Fubini's theorem
\begin{align*}
	\int_{\Omega} r_{n}(x)^{-d-\delta p/q} &|f|_{W_p^s(K_n^\circ(x))}^{p}\dd x\\
	&= \sum_{\abs{\alpha}=s}\int_{\Omega}\abs{D^{\alpha}f(y)}^{p} \int_{\Omega}r_{n}(x)^{-d-\delta p/q} 1(y\in K_{n}(x))\dd x\dd y\\
	&\le\, \sup_{y\in\Omega}\int_{\Omega}r_{n}(x)^{-d-\delta p/q} 1(y\in K_{n}(x))\dd x
\end{align*}
since $f$ has semi-norm at most one.
\smallskip

Let $y\in \Omega$ be arbitrary. Note that $y\in K_{n}(x)$ implies $\norm{y-x}\leq r_n(x)$, and thus, since $r_n(x)\geq n^{-1/d}$,
\begin{multline*}
	\int_{\Omega}r_{n}(x)^{-d-\delta p/q} 1(y\in K_{n}(x))\dd x
	\,\le \int_{\Omega}\frac{1}{\max\{n^{-1/d},\norm{y-x}\}^{d+\delta p/q}}\,\text{d}x\\
	\le\int_{\norm{y-x}\leq n^{-1/d}}n^{1+\delta p/qd}\,\text{d}x
	\,+\,\int_{\norm{y-x}>n^{-1/d}}\norm{y-x}^{-d-\delta p/q}\text{d}x.
\end{multline*}
where we extended the integral to the whole of $\IR^d$.
The first integral on the right is $\omega_d n^{\delta p/(dq)}$, where $\omega_{d}$ is the volume of the $d$-dimensional unit ball, and the right integral transforms, by integration in spherical coordinates and subsequent substitution $v=n^{1/d}u$, to
\[
d\omega_{d} \int_{n^{-1/d}}^{\infty} u^{-d-\delta p/q}u^{d-1}\text{d}u
	\,=\, d\omega_{d}n^{\delta p/(dq)}\int_1^{\infty} v^{-1-\delta p/q}\text{d}v
	\,\preccurlyeq\, n^{\delta p/(dq)}.
\]
Note that the surface area of $\mathbb{S}^{d-1}$ equals $d\omega_d$.
This shows
\[
	\int_{\Omega}r_{n}(x)^{-d-\delta p/q} 1(y\in K_{n}(x))\dd x
	\,\preccurlyeq\, n^{\delta p/	(dq)},
\]
where the implicit constant depends only on $d,\delta,p$ and $q$. Consequently, taking the supremum over all $y\in\Omega$ yields
\[
	\int_{\Omega} r_{n}(x)^{-d-\delta p/q} |f|_{W_p^s(K_n^\circ(x))}^{p}\dd x
	\,\preccurlyeq\, n^{\delta p/(dq)}.
\]
Therefore, 
\[
 	\sup_{\Vert f\Vert_{W_p^s(\Omega)} \le 1}\Big\Vert r_n(x)^{s-d/p}\, |f|_{W_p^s(K_n^\circ(x))} \Big\Vert_{L_q(\Omega)}
	\le \brackets{\int_{\Omega}r_{n}(x)^{(sq+\delta)r}\dd x}^{1/rq} n^{\delta/(dq)}.
\]
To establish \eqref{eq:det-new} we note that by Jensen's inequality and Fubini's theorem we have
\[
	 \IE \sup_{\Vert f\Vert_{W_p^s(\Omega)} \le 1} \left\Vert r_n(x)^{s-d/p}\, |f|_{W_p^s(K_n^\circ(x))} \right\Vert_{L_q(\Omega)} 
	 \le \brackets{\int_{\Omega}\IE\, r_{n}(x)^{(sq+\delta)r}\dd x}^{1/rq}n^{\delta/(dq)}.
 \]
 By Proposition~\ref{prop:localradius} we have that $ \IE\, r_{n}(x)^{(sq+\delta)r} \preccurlyeq n^{-sqr/d-\delta r/d}$  with a constant independent of $x$ and this completes the proof in this case.$ \hfill \square$

\subsubsection{The case $q\ge p$}
We use
$r_n(x) \le c_6 h_{P_n,\Omega}$, so that
\begin{multline*}
 \left\Vert r_n(x)^{s-d/p}\, |f|_{W_p^s(K_n^\circ(x))} \right\Vert_{L_q(\Omega)}
 \,\le\, (c_6 h_{P_n,\Omega})^{s-d/p} \left\Vert |f|_{W_p^s(K_n^\circ(x))} \right\Vert_{L_q(\Omega)}
  \\
 \le\, (c_6 h_{P_n,\Omega})^{s-d/p} \left\Vert |f|_{W_p^s(K_n^\circ(x))} \right\Vert_{L_p(\Omega)}^{p/q}
 \cdot \underbrace{\left\Vert |f|_{W_p^s(K_n^\circ(x))} \right\Vert_{L_\infty(\Omega)}^{1-p/q}}_{\le\, 1}
\end{multline*}
Now, since $K_n(x) \subset B(x,c_6 h_{P_n,\Omega})$, we have that
\begin{multline*}
\left\Vert |f|_{W_p^s(K_n^\circ(x))} \right\Vert_{L_p(\Omega)}^{p/q}
\,\le\, \Bigg(\sum_{|\alpha|=s} \int_\Omega \int_\Omega \left|D^\alpha f(y)\right|^p \, 1(y\in B(x,c_6 h_{P_n,\Omega})) \, \dd y\, \dd x\Bigg)^{1/q} \\
=\, \Bigg(\sum_{|\alpha|=s} \int_\Omega  \left|D^\alpha f(y)\right|^p \underbrace{\int_\Omega 1(x\in B(y,c_6 h_{P_n,\Omega})) \, \dd x}_{\preccurlyeq\, h_{P_n,\Omega}^d}\, \dd y\Bigg)^{1/q}
\,\preccurlyeq\, h_{P_n,\Omega}^{d/q}.
\end{multline*}
Together,
\[
 \sup_{\Vert f\Vert_{W_p^s(\Omega)} \le 1} 
 \left\Vert r_n(x)^{s-d/p}\, |f|_{W_p^s(K_n^\circ(x))} \right\Vert_{L_q(\Omega)}
 \,\preccurlyeq\, h_{P_n,\Omega}^{s-d/p+d/q}.
\]
Taking the expected value,
the estimate \eqref{eq:det-new} follows 
from \cite[Theorem~2.1]{RS16}.
Again we use the cone condition of $\Omega$.
$\hfill \square$

\begin{rem}\label{rem:optimal-points-proof}
We obtained in the case $q\ge p$ the upper bound
\[
 \sup_{\norm{f}_{W^{s}_{p}(\Omega)}\le 1}\norm{f-A_n(f)}_{L_{q}(\Omega)}
 \,\preccurlyeq\, h_{P_n,\Omega}^{s-d/p+d/q},
\]
which holds true for any realization of the point 
set $P_n$ satisfying Scenario~2.
By inserting a point set
with covering radius of order $n^{-1/d}$
instead of iid points,
we obtain the order of convergence $n^{-s/d-1/p+1/q}$.
We thus obtain the order of convergence
of optimal sampling points as stated in~\eqref{eq:UBOPTdet}
for all open and bounded domains with the cone condition
as a byproduct.
Note that the upper bound in the case $q< p$ is implied by~\eqref{eq:det-old},
and the corresponding lower bounds are known,
see Remark~\ref{rem:LB}.
\end{rem}

\section{Lower bounds}
\label{sec:lower}

In this section we complete the proof of Theorem~\ref{thm:main} 
by providing the corresponding lower bound.
As noted above, see Remark~\ref{rem:LB}, the only case that is open is the lower bound for $p=q=\infty$. 
We use a technique of Bakhvalov: bump functions, 
the average error with respect to a discrete measure and the 
theorem of Fubini.

\smallskip

There are bump functions $f_1,\hdots,f_m$
with support in disjoint balls contained in $\Omega$ of volume 
of the order $m^{-1}$ 
such that $\Vert f_i \Vert_{W_\infty^s(\Omega)} =1$
and $\Vert f_i \Vert_\infty \succcurlyeq m^{-s/d}$;
we obtain such functions by a simple scaling. 
Now we consider the finite set
\[
 F \,:=\, \left\{ \sum_{i=1}^m \varepsilon_i f_i \,\Big|\, \varepsilon_i \in \{-1,1\} \right\}.
\]
Then $\Vert f \Vert_{W_\infty^s(\Omega)} =1$ for all $f\in F$.
To distinguish between $f$ and $g$ for each pair of distinct functions $f,g \in F$
we would need 
at least $m$ well chosen function values and here we loose a log term since we 
have to use iid information. 
\smallskip

Let $m \approx n/ 2\log n$.
For every algorithm $S_{P_n}$
that uses the random sampling point set $P_n$,
we have
\begin{multline*}
\err(S_{P_n},W_\infty^s(\Omega),L_\infty(\Omega)) \,=\,
\sup_{\Vert  f \Vert_{W_\infty^s}  \le 1 } \IE\, \Vert f - S_{P_n}(f) \Vert_\infty\\
\ge\, \sup_{f \in F }\, \IE\, \Vert f - S_{P_n}(f) \Vert_\infty
\,\ge\,  \IE\,\frac{1}{2^m} \sum_{f\in F} \Vert f - S_{P_n}(f) \Vert_\infty.
\end{multline*}
With constant probability, $P_n$ misses 
one of the balls  
(recall the coupon collector's problem)
and $S_{P_n}$ cannot determine
the sign~$\varepsilon_i$ of the corresponding bump~$f_i$.
Thus
\[
 \Vert f - S_{P_n}(f) \Vert_\infty
 \,\ge\, \Vert f_i \Vert_\infty
\]
for at least half of the functions $f\in F$ and
\[
 \frac{1}{2^m} \sum_{f\in F} \Vert f - S_{P_n}(f) \Vert_\infty
 \,\ge\, \frac12\, \Vert f_i \Vert_\infty.
\]
This yields
\[
\err(S_{P_n},W_\infty^s,L_\infty) \,\succcurlyeq\, m^{-s/d}.
\]
$\hfill \square$

\section{Integration}

The results on the integration problem are obtained as follows.

\medskip

Corollary~\ref{cor1} is a direct consequence of (the upper bound of) 
Theorem~\ref{thm:main}, together with the lower bound that holds for 
\emph{all} randomized algorithms. 
We use the algorithm for $n$ function values for $L_2$-approximation, 
as in the proof of the theorem, to construct an approximation 
$f_n$. Then we use another $n$ function values 
for the 
standard Monte Carlo method for $(f-f_n)$, 
to obtain the additional order $n^{-1/2}$. 
This algorithm (``control variates'' or ``separation of the main part'') 
clearly uses iid random points, uniformly 
distributed in $\Omega$. 
\medskip

Corollary~\ref{cor2} is a direct consequence of Theorem~\ref{thm:main-det},
together with the lower bounds of~\cite{KS20}.
We use the algorithm for $n$ function values for $L_1$-approximation, 
as in the proof of the theorem, to construct an approximation 
$f_n$. Then we output the integral of~$f_n$.

\section{Measurability}
\label{sec:meas}

In order to guarantee the measurability of our algorithm~$A_n$, we need that the direction $\xi(x)$ of the cone $K(x)$ in the definition of the interior cone condition of $\Omega$ depends continuously on the apex $x$ at almost every point. The following lemma provides us with such a choice if we take slightly smaller cones.
\begin{lemma} \label{lem:ccones}
	Let $\Omega$ be open and bounded satisfying an interior cone condition with parameters $r>0$ and $\theta\in (0,\pi)$. Then we can find for every $x\in\Omega$ a cone with apex $x$, radius $(c_{\theta}/2)r$, angle $2\arcsin(c_{\theta}/4)$ and direction $\xi(x)\in\mathbb{S}^{d-1}$ such that $\xi$ is continuous almost everywhere on $\Omega$. Here, $c_{\theta}$ is as in Lemma~\ref{lem:conecone}.
\end{lemma}
\begin{proof}
	We find finitely many sets $\Omega_k$, $k=1,\ldots,K$, each star-shaped with respect to a ball $B_k=B(z_k,r_k)\subset \Omega_k$ (i.e., the straight line connecting any $x\in\Omega_k$ with any  $ y\in B_k$ is contained in $\Omega_k$), such that their union is $\Omega$. For the existence of such a family of sets see e.g.\ \cite[1.1.9, Lemma 1]{Maz11} or \cite[Lemma 11.31]{W04}. 

	Fix some $k$ as above. We can define $\xi(x)$ for every $x\in \Omega_k$ such that 
	$\xi$ is continuous almost everywhere on the interior of $\Omega_k$.  Namely, from the proof of \cite[Lemma 11.31]{W04} (as well as the proofs of Proposition 11.26 and Lemma 3.10 there) one can deduce that for every $k$ one can choose $r_k=(c_{\theta}/2)r$ and for every $x\in\Omega_k\setminus \{z_k\}$ a cone with apex $x$, radius $r_0$, angle $2\arcsin(c_{\theta}/4)$ and direction $\xi(x)=(z_k-x)/\|z_k-x\|_2$ which is continuous on the interior of $\Omega_k$ except for the point $z_k$. 
	Note that $2\arcsin(c_{\theta}/4)\le \pi/3$.

In order to define $\xi$ on $\Omega$, we set $\Omega_1'=\Omega_1$ and $\Omega_k'=\Omega_k\setminus \bigcup_{j<k}\Omega_k'$ for $2\le k\le K$. Then any $x\in\Omega$ belongs to exactly one $\Omega_k'$. Using this, we can define $\xi(x)$ everywhere on $\Omega$. 
The function $\xi$  is continuous on $E_1^C$,
where $E_1$ is the union of the points $z_k$ and the boundaries $\partial \Omega_k$.
Since for every $x\in\partial\Omega_k$ and every $r<r_k$
the ball $B(x,r)$ contains a ball of proportional volume which is contained in
the interior of $\Omega_k$ (and thus not in $\partial\Omega_k$),
the density of $\partial \Omega_k$ is smaller than one at every $x\in \partial \Omega_k$.
Thus, by Lebesgue's density theorem, the sets $\partial \Omega_k$ and $E_1$ are null sets.
Note that $E_1^C$ is also an open set.
\end{proof}

Possibly modifying constants we assume from now on that $\Omega$ satisfies an interior cone condition with parameters $r$ and $\theta$ such that the direction $\xi(x)$ of the cone $K(x)$ is continuous in $x\in\Omega$ except from a set $E_1\subset \Omega$ with 
measure zero. 
We now prove the measurabilty of the functions appearing
in the previous sections. Recall the definitions of $m_P(x)$, $r_P(x)$, $r_P^*(x)$, $K_P(x)$ and $A_P f(x)$
from Algorithm~\ref{alg} and~\eqref{eq:defrnast}.
The measurability of the corresponding random quantities follows from the following proposition.

\begin{prop}\label{pro:measurability}
Let $n\in\IN$. Then the following mappings are measurable (with respect to the corresponding Lebesgue or Borel $\sigma$-algebras):
\begin{enumerate}[label=\roman*)]
	\item $m_P(x)$ and consequently 
	$r_P(x)$ as functions of $(x,P)\in\Omega^{n+1}$, 
	\item $r_P^*(x)$ as a function of $(x,P)\in\Omega^{n+1}$,
	\item $A_Pf(x)$ as a function of $(x,P)\in\Omega^{n+1}$ for every $f\in W^s_p(\Omega)$,
	\item $\sup_{\|f\|_{W^s_p(\Omega)}\le 1} \|f-A_Pf\|_{L_q(\Omega)}$ as a function of $P\in\Omega^n$,
	\item $|f|_{W^s_p(K_P^\circ(x))}$ as a function of $(x,P)\in\Omega^{n+1}$ for every $f\in W^s_p(\Omega)$.
\end{enumerate}
\end{prop}

Note that we slightly abuse notation since we allow repeated points in the $n$-point set $P$
such that the functions are defined for all $(x,P)\in\Omega^{n+1}$. 
In the following, we view $\Omega^{k},k\in\IN,$ as a subset of $\IR^{dk}$ from which it inherits the Euclidean distance. Since, on a finite-dimensional space, all norms are equivalent, we can also use distances induced by other norms.
Before we provide a proof we need a few lemmas. We consider $n\in\IN$ arbitrary but fixed. 

\begin{lemma}
	\label{lem:hpn-continuous}
	The function $P\mapsto h_{P,\Omega} $ is continuous on $\Omega^n$. 
\end{lemma}
\begin{proof}
The function is even Lipschitz-continuous. It is sufficient to check that
\[
h_{\{x_1,\hdots,x_n\},\Omega} 
\,\ge\, h_{\{y_1,\hdots,y_n\},\Omega}  
- \max_{i=1\hdots n} \Vert x_i - y_i \Vert_2,
\]
for all point sets $\{x_1,\hdots,x_n\} \subset \Omega$ and $\{y_1,\hdots,y_n\} \subset \Omega$.
\end{proof}
\begin{lemma} \label{lem:boundary}
The set $E_2$ is a null set, where
\[
E_2:=\left\{(x,P)\in\Omega^{n+1}: P\cap \partial K(x,2^{-m}r)\neq \emptyset \text{ for some }m\in\IN\right\}.
\]
\end{lemma}
\begin{proof}
Since $E_2$ is measurable, it is sufficient to show for any fixed $x\in\Omega$ and $m\in\IN$
that the set of all $P\in \Omega^n$ with $P\cap \partial K(x,2^{-m}r)\neq \emptyset$
is a null set.
But this is evident from the fact that the boundary of a cone has measure zero.
\end{proof}

\begin{lemma}
	\label{lem:hm-continuous}
For all $m\le m_0$, the function 
\[h_m\colon \Omega^{n+1}\to\IR, \quad h_m(x,P)=h_{P\cap K(x,2^{-m}r),K(x,2^{-m}r)}\] 
is continuous on $E_1^C \cap E_2^C$. 
In particular, 
the function is measurable.
\end{lemma}

\begin{proof}
	Let $(x,P)\in E_1^C \cap E_2^C$ and let $\varepsilon>0$. We can perturb both $x$ and $P$ by a small amount such that each point of $P$ is delocated by less than $\varepsilon/2$ without removing or adding any point of $P$ to the cone $K(x,2^{-m}r)$ and thus the covering radius of the perturbed configuration cannot change by more than $\varepsilon$. We provide the details. Since $(x,P)\in E_2^C$, we find $\delta>0$ such that for all $n$-point sets $Q=\{y_1,\ldots,y_n\}\subset \Omega$ with $\|y_i-x_i\|<\delta$ for all $i$ we have that $\dist(Q,\partial K(x,2^{-m}r))>\delta$. Since $\xi$ is continuous on $E_1^C$, we find $\delta'>0$ small enough such that for every $y\in B(x,\delta')$ the sets $\partial K(x,2^{-m}r)$ and $\partial K(y,2^{-m}r)$ are closer than $\delta/2$ in Hausdorff distance. 
	
By the triangle inequality, for every $(y,Q)$ as above we have that for every point in $Q\cap K(y,2^{-m}r)$ we find a point in $P\cap K(x,2^{-m}r)$ which has distance less than $\delta$ from it,
and vice versa. If $\delta<\varepsilon/2$, then $|h_m(y,Q)-h_m(x,P)|<\varepsilon$ concludes the proof.
Note that we also showed that $E_1^C\cap E_2^C$ is open.
\end{proof}
\begin{lemma} \label{lem:levelset}
	For all $m\le m_0$ and $c>0$, 
	the set $\{h_m=c\}$ 
	is a null set.
In particular,
	 \[E_3 := \{(x,P)\in\Omega^{n+1} \colon h_m(x,P) = c_1 2^{-m} r \text{ for some } m\le m_0\}\]
	 is a null set.
	 Moreover, the set $\{P\in\Omega^n \colon h_{P,\Omega}=c_0r\}$ is a null set.	
\end{lemma}
\begin{proof}
If the level set $\{h_m=c\}$ is empty, this is clear, so assume it is not empty. By Lemma~\ref{lem:hm-continuous} it is a measurable set and it is therefore sufficient to show that for every $x\in \Omega$ the set
\[
E=\{P\in\Omega^n: h_m(x,P)=c\}
\]
is a null set. For this we will use the Lebesgue density theorem as follows. Let $P\in E$ be arbitrary with $P=\{x_1,\ldots,x_n\}$. Choose $\varepsilon_0>0$ small enough such that $B(y,\varepsilon_0)\subset \Omega$ for every $y\in P$. Then for any $\varepsilon<\varepsilon_0$, 
\[
U(P,\varepsilon):=B(x_1,\varepsilon)\times \cdots \times B(x_n,\varepsilon)\subset \Omega^n.
\]
Since $h_m(x,P)=c$ and $h_m(x,P)$ is nothing but the maximal value
of the continuous function $\dist(\cdot,P \cap K)$
on the compact set $K := K(x,2^{-m}r)$,
we find a point $z\in  K$ 
such that $\dist(z,P \cap K)=c$.
Thus, $B(z,c)$ is empty of $P \cap K$.
For each $i$, we set $B_i=\overline{B(z,c)}$ if $x_i\in K$
and $B_i=K$ if $x_i\not\in K$.
In both cases, we have $\vol(B(x_i,\varepsilon)\setminus B_i)\ge 1/2 \cdot\vol(B(x_i,\varepsilon))$.
Moreover, every point of $B(x_i,\varepsilon)\setminus B_i$ is either outside the cone $K$
or at a distance greater than $c$ from $z$.
Thus, for any point set
\[
Q \in (B(x_1,\varepsilon)\setminus B_1)\times \cdots \times (B(x_n,\varepsilon)\setminus B_n)
\]
we have $h_m(x,Q)>c$,
meaning that $Q\not\in E$.
This gives
\[
\frac{\vol(U(P,\varepsilon)\cap E)}{\vol(U(P,\varepsilon))}\le 1-2^{-n}\quad \text{for all }0<\varepsilon<\varepsilon_0,
\]
Thus, the density of $E$ at $P$ is not equal to one.
For the last statement,
note that it does not matter whether the density is defined via
the test sets $U(P,\varepsilon)$ or classical Euclidean balls.
Since the density is not equal to one for all $P\in E$,
the Lebesgue density theorem shows
that $E$ and consequently $\{h_m=c\}$ is a null set.
In the same way one can show that the level sets of the function 
$P\mapsto h_{P,\Omega}$ are null sets.
The set $E_3$ is a null set, since it is a finite union of level sets $\{h_m=c\}$.
\end{proof}

\begin{lemma}\label{lem:aec}
The function $m_P(x)$ is continuous at almost every point $(x,P) \in \Omega^{n+1}$.
\end{lemma}

\begin{proof}
	Let $(x,P)\in E_1^C\cap E_2^C\cap E_3^C$. Then $h_m(x,P)\neq c_1 2^{-m}r$ for every $m\in\IN$ and $h_m$ is continuous at $(x,P)$. We want to deduce from this that $m_Q(y)=m_P(x)$ for all $(y,Q)$ in a neighbourhood around $(x,P)$. This implies that $m_P(x)$ is locally constant and thus continuous at $(x,P)$, which yields the conclusion. 
	
The condition $m_P(x)=0$ is equivalent to $h_{m}(x,P)>c_1 2^{-m}r$ for all $m\in\{1,\ldots,m_0\}$. The continuity of the finitely many $h_m$, $m\in\{1,\ldots,m_0\}$, implies that also $m_Q(y)=0$ for $(y,Q)$ in a neighbourhood of $(x,P)$.

The condition $m_P(x)=m\in\{1,\ldots,m_0\}$ is equivalent to $h_m(x,P)<c_1 2^{-m}r$ and $h_{m'}(x,P)>c_1 2^{-m'}r$ for all $m'\in\{m+1,\ldots,m_0\}$. For $m_P(x)=m_0$ the latter set is empty. Again, the continuity of the finitely many $h_m$, $m\in\{1,\ldots,m_0\}$, implies that also $m_Q(y)=m$ for 
$(y,Q)$ in a neighbourhood of $(x,P)$.
\end{proof}

We can now give the proof of Proposition~\ref{pro:measurability}.

\begin{proof}[Proof of Proposition~\ref{pro:measurability}]

By Lemma~\ref{lem:aec}, the function $m_P(x)$
and therefore also $r_P(x)$ is continuous almost everywhere. 
This proves (i). 
Since the radius $r_P(x)$ and the direction $\xi$ are continuous
almost everywhere, also the function $(x,P)\mapsto K_P^\circ(x)$ is continuous in the Hausdorff distance. This proves~(v). 

We turn to (ii).
For almost all $(y,x,P)\in\Omega^{n+2}$ we have that $y$
does not lie on boundary of $K_P(x)$ and that $K_P(x)$
is continuous at $(y,x,P)$ in the Hausdorff distance.
Thus, the value of the indicator function $1(y\in K_P(x))$
is constant in a neighborhood of $(y,x,P)$.
This means that the function $\Omega^{n+2}\to\IR$,
$(y,x,P)\mapsto 1(y\in K_P(x))$ is measurable.
Moreover, we already know that $(y,x,P)\mapsto r_P(x)$ is measurable.
Thus, also $r_P^*(y)$ is measurable as essential supremum of the product 
of these two functions over $x\in\Omega$.

Next we prove (iii). Let $f\in W^s_p(\Omega)$. 
We show that $A_Pf(x)$ is continuous almost everywhere as a function of $(x,P)\in\Omega^{n+1}$.
If $(x,P)$ is such that $h_{P,\Omega}>c_0 r$,
the output $A_Pf(x)$ is continuous at $(x,P)$
since it equals zero in a whole neighborhood of $(x,P)$.
The set of all $(x,P)$ with $h_{P,\Omega}=c_0 r$ is a zero set
and may be ignored.
So let now $(x,P)$ be such that $h_{P,\Omega}<c_0 r$.
Then we are in Scenario~2 of our algorithm
in a whole neighborhood of $(x,P)$. 
By \cite[Lemma~7]{KS20} and \cite[Theorem~4.7]{W04},
the output $A_Pf(x)$ 
can be computed as follows:
First we compute the solutions $a_j^*(x)$ to formulas (4.6) and (4.7)
in~\cite[Corollary~4.4]{W04}, which depend
continuously on $x$, the involved point set and the parameter $\delta$.
In our case, the involved point set is given as a certain subset
$Q(x,P)$ of the point set $P\cap K_P(x)$,
selected according to the proof of \cite[Lemma~7]{KS20},
and the parameter $\delta=\delta(x,P)$ is a constant multiple 
of the covering radius of $Q(x,P)$ in $K_P(x)$.
Note that the recursive selection procedure from \cite[Lemma~7]{KS20}
is not completely specified, but it may easily be specified e.g.\ by choosing
the point with the smallest index whenever there are multiple choices.
Both $Q(x,P)$ and $\delta(x,P)$ depend continuously on $(x,P)$
for almost all $(x,P)$.
Secondly, we put $A_Pf(x)=\sum_j a_j^*(x) f(x_j)$,
which depends continuously on $(x,P)$ whenever the $a_j^*$ do so.

We prove (iv). We can replace the supremum over the unit ball by a
countable supremum and then the statement follows from (iii). We provide the details.
There is a countable subset $S$ of the unit ball of $W_p^s(\Omega)$
which is dense with respect to the supremum norm:
In all cases except $p=1$ and $s=d$, the unit ball is relatively compact
in $C_b(\Omega)$. Thus, for any $k\in\IN$ it may be covered by finitely many
balls with radius $1/k$ and center inside the unit ball,
and we obtain $S$ as the union of the centers over all $k\in\IN$.
In the case $p=1$ and $s=d$, the space $W_p^s(\Omega)$
is separable and we get a countable subset $S$ of the unit ball 
which is dense with respect to the Sobolev norm,
and therefore also with respect to the supremum norm.
Now, for every $f$ in the unit ball
and every $\varepsilon>0$, we find some $g\in S$ with $\|f-g\|_{C_b(\Omega)}
<\varepsilon$. The linearity and boundedness of $A_P:C_b(\Omega)\to B(\Omega)$ give
\[
\|(f-A_P f)-(g-A_P g)\|_{B(\Omega)}
\le \|f-g\|_{C_b(\Omega)} + \|A_P(f-g)\|_{B(\Omega)}
<(1+c_2)\varepsilon
\]
and thus 
\[
 \big| \|f-A_P f\|_{L_q(\Omega)} - \|g-A_P g\|_{L_q(\Omega)} \big|
 \,\le\, \vol(\Omega)^{1/q} (1+c_2)\varepsilon
\]
and hence
the supremum over the unit ball may be replaced by the supremum over~$S$.
\end{proof}

\subsection*{Acknowledgement}
D.~Krieg and M.~Sonnleitner are supported by the Austrian Science Fund (FWF) Project F5513-N26, 
which is a part of the Special Research 
Program \emph{Quasi-Monte Carlo Methods:~Theory and Applications}.

\bibliographystyle{plain}

\end{document}